\documentclass{article}

\usepackage[final]{graphicx} 
\usepackage{subcaption}    
\usepackage{color} 
\usepackage{varioref}  
\usepackage{rotating}  

\usepackage{pdflscape}		
\usepackage{setspace}		
\usepackage{mathrsfs}			
\usepackage{mathtools}
\usepackage{verbatim}			

\usepackage[final]{listings} 
\usepackage[pdftex,hidelinks]{hyperref}
\usepackage{cite} 

\usepackage{tikz,ulem} 
\usetikzlibrary{calc}
\usetikzlibrary{patterns}
\usetikzlibrary{positioning}
\usetikzlibrary{decorations.markings}
\usetikzlibrary{decorations.pathmorphing}
\usetikzlibrary{arrows}

\usepackage{latexsym}
\usepackage{amsfonts}
\usepackage{amssymb}
\usepackage{amsmath}
\usepackage{amsthm}
\usepackage{bbm}
\usepackage{afterpage}
\usepackage{fancyhdr}
\usepackage{lipsum} 
\usepackage{amsfonts}
\usepackage{amsthm}
\usepackage{amsmath}
\usepackage{amscd}
\usepackage{amssymb}
\usepackage{amsbsy}
\usepackage{epsf,graphicx,amsmath,verbatim,epsfig,amssymb,bbm}
\usepackage{tikz-cd,tikz, MnSymbol, relsize,enumitem, dsfont}
\usepackage{enumitem}
\usepackage{amsthm}

\usepackage[all]{xy}

\usepackage{pgfplots, tikz}

\usepackage{pgfplots}
\usepackage{tkz-euclide}
\usepgfplotslibrary{fillbetween}

\pgfplotsset{compat=1.6}

\pgfplotsset{soldot/.style={color=blue,only marks,mark=*}} \pgfplotsset{holdot/.style={color=blue,fill=white,only marks,mark=*}}

\newtheorem{theorem}{Theorem}[section]

\newtheorem{lemma}[theorem]{Lemma}
\newtheorem{corollary}[theorem]{Corollary}

\theoremstyle{definition}
\newtheorem{definition}[theorem]{Definition}

\newtheorem{convention}[theorem]{Convention}

\theoremstyle{remark}

\newcommand{\RR}{{\rm I\kern -1.6pt{\rm R}}}

\def\deq{:=}

\def\R{\mathbb{R}}

\def\id{\mathbbm{1}}

\def\H{\mathbb{H}}

\DeclareMathOperator{\Ima}{Im}

\def\cc{Carnot-Carath\'{e}odory~}

\def\H{\mathbb{H}}

\newcommand{\pilip}[1]{\pi_{#1}^{\text{Lip}}}

\newcommand{\length}[1]{\ell}

\def\HH{\mathbb H}

\def\NN{\mathbb N}
\def\RR{\mathbb R}

\newcommand{\lip}[1]{\text{Lip}(#1)}

\def\lmin{\ell_{\text{min}}}

\newcommand{\core}[1]{{#1}_\infty}

\DeclareMathOperator{\Lip}{Lip}

\DeclareMathOperator{\diam}{diam}

\newcommand\blfootnote[1]{%
  \begingroup
  \renewcommand\thefootnote{}\footnote{#1}%
  \addtocounter{footnote}{-1}%
  \endgroup
}

\begin{document}

\title{Existence of length minimizers in homotopy classes of Lipschitz paths in $\mathbb{H}^1$}
\author{Daniel Perry}
{\let\newpage\relax\maketitle}

  \begin{abstract}
We show that for any purely 2-unrectifiable metric space $M$, for example the Heisenberg group $\HH^1$ equipped with the \cc metric, every homotopy class $[\gamma]$ of Lipschitz paths contains a length minimizing representative $\core{\gamma}$ that is unique up to reparametrization. 
The length minimizer $\core{\gamma}$ is the core of the homotopy class $[\gamma]$ in the sense that the image of $\core{\gamma}$ is a subset of the image of any path contained in $[\gamma]$. Furthermore, the existence of length minimizers guarantees that only the trivial class in the first Lipschitz homotopy group of $M$ with a base point can be represented by a loop within each neighborhood of the base point. The results detailed here are used in \cite{perry2024universal} to define and prove properties of a universal Lipschitz path space over $\HH^1$.
 \blfootnote{{\bf Mathematical Reviews subject classification.} Primary: 53C17, 28A75 ; Secondary: 57K33, 54E35 \\ {\bf Acknowledgments.} This material is based upon work supported by the National Science Foundation under Grant Number DMS 1641020.  The author was also supported by NSF awards 1507704 and 1812055 and ARAF awards in 2022 and 2023.}
  \end{abstract}


\normalem

\section{Introduction}

%

In this paper, we prove the following theorem.

\begin{theorem}\label{main result}
Let $M$ be a purely 2-unrectifiable metric space, for example the Heisenberg group $\H^1$ endowed with the \cc metric. For any homotopy class $[\gamma]$ of Lipschitz paths in $M$, there exists a length minimizing Lipschitz path $\core{\gamma}\in[\gamma]$ where
\[
\length{M}(\core{\gamma})=\inf\{\length{M}(\gamma)~|~\gamma\in[\gamma]\}.
\]
Moreover, for any representative $\gamma\in[\gamma]$ in the class, $\Ima(\core{\gamma})\subset\Ima(\gamma)$.
\end{theorem}
A length minimizer $\core{\gamma}\in[\gamma]$ can be thus thought of as the core of the homotopy class $[\gamma]$ where the extraneous branches of the paths in the class have been pruned. An immediate consequence is that for every point in a purely 2-unrectifiable metric space, only the trivial class in the first Lipschitz homotopy group can be represented by a loop within each neighborhood of the point (Corollary~\ref{heisenberg non-singular points}). 

Studying metric spaces, in particular Heisenberg groups endowed with a \cc metric, through Lipschitz homotopies was introduced in \cite{Dej} with the definition Lipschitz homotopy groups. Since, Lipschitz homotopy groups have been calculated for various sub-Riemannian manifolds in \cite{Dej}, \cite{Haj}, \cite{HajSch}, \cite{HajTys}, \cite{perry2020lipschitz}, and \cite{Weg}. For an overview of sub-Riemannian geometry, see \cite{Mon}.

Results in \cite{Dej}, \cite{perry2020lipschitz}, and \cite{Weg} concerning the Lipschitz homotopy groups of the Heisenberg group $\H^1$ and contact 3-manifolds rely heavily on these sub-Riemannian manifolds endowed with the \cc metric being purely 2-unrectifiable metric spaces in the sense of \cite{Amb}. As is shown in \cite{Amb}, the Heisenberg group $\H^1$ is purely 2-unrectifiable and, as is shown in \cite{perry2020lipschitz}, any contact 3-manifold endowed with a sub-Riemannian metric is purely 2-unrectifiable. We likewise make significant use of results of purely 2-unrectifiable metric spaces to conclude the existence of length minimizers in homotopy classes.

The key step to the proof of Theorem~\ref{main result} is, given a Lipschitz path $\gamma$ in a homotopy class $[\gamma]$, determining a sequence of Lipschitz paths that are uniformly Lipschitz homotopic to $\gamma$ and whose lengths converge to the infimum. Once such a sequence is obtained, Arzel\`{a}-Ascoli theorem yields a length minimizing path $\core{\gamma}$ together with a homotopy to $\gamma$. 

To find this sequence of paths, we apply a lemma (Lemma~\ref{Desirable homotopy}) which, given any homotopy from $\gamma$ to any Lipschitz path $\beta$, constructs a homotopy with controlled Lipschitz constant from $\gamma$ to a shorther path $\beta'$. The proof of Lemma~\ref{Desirable homotopy} relies on a result of Wenger and Young in \cite{Weg} concerning a factorization of Lipschitz maps with purely 2-unrectifiable target though a metric tree. Their result is stated in Theorem~\ref{Wenger and Young}.

The results in this paper are utilized in \cite{perry2024universal} to define a universal Lipschitz path space over $\H^1$ and prove requisite properties. Of note, Lemma~\ref{Desirable homotopy} and Theorem \ref{uniqueness of core} illustrate that purely 2-unrectifiable spaces are tree-like homotopy spaces. As is reported in Corollary 4.8 in \cite{perry2024universal}, the universal Lipschitz path space over a tree-like homotopy space is a Lipschitz simply connected length space which satisfies a unique lifting property.



\medskip
\noindent {\bf Acknowledgment}. The author wishes to thank Chris Gartland for his invaluable input on this paper, in particular his suggestion to use Arzel\`{a}-Ascoli theorem. In additon to Gartland, the author would also like to thank Fedya Manin, David Ayala, Lukas Geyer, and Carl Olimb for their input, thoughts, and support as this paper came together.

\section{Background}


\subsection{Homotopy, geodesics, and metric trees}

\begin{convention}
Throughout this paper, $I=[0,1]$ is the closed interval endowed with the Euclidean metric. All paths will have domain $I$. For a metric space $M$ and a path $\gamma:I\rightarrow M$, the length of the path $\gamma$ is denoted $\length{M}(\gamma)$.  For metric spaces $A$ and $M$, the Lipschitz constant of a Lipschitz function $f:A\rightarrow M$ will be denoted by $\Lip(f)$. 
\end{convention}

As we proceed, we will endow $I\times I$ with the $L^1$ metric: for $(s,t),(s',t')\in I\times I$,
\[
d^1((s,t),(s',t'))=|s-s'|+|t-t'|.
\]
The metric $d^1$ is Lipschitz equivalent to the Euclidean metric on $I\times I$.

\begin{definition}\label{path classes}
Let $M$ be a metric space. Two Lipschitz paths $\gamma,\gamma':I\longrightarrow M$ are \emph{homotopic rel endpoints} if the initial points $\gamma(0)=\gamma'(0)$ and end points $\gamma(1)=\gamma'(1)$ of the paths agree and there exists a Lipschitz map $H:I\times I\rightarrow M$ such that
\[
H|_{I\times\{0\}}=\gamma,\hspace{.25cm} H|_{I\times\{1\}}=\gamma',\hspace{.25cm} H|_{\{0\}\times I}=\gamma(0),\hspace{.25cm}\text{and}\hspace{.25cm} H|_{\{1\}\times I}=\gamma(1).
\]
The map $H$ is a \emph{homotopy} from $\gamma$ to $\gamma'$. For a Lipschitz path $\gamma$, the class of all Lipschitz paths homotopic rel endpoints to $\gamma$ is denoted $[\gamma]$ and is referred to as the \emph{homotopy class} of $\gamma$. 
\end{definition}

The homotopy classes of loops based at a point $x_0\in M$ are the elements of the first Lipschitz homotopy group $\pilip{1}(M,x_0)$ of the metric space $M$. For the complete definition of Lipschitz homotopy groups and the initial study of $\pilip{1}(\H^1)$, see \cite{Dej}. Another example of studying first Lipschitz homotopy groups of purely 2-unrectifiable metric spaces can be found in \cite{perry2020lipschitz} where contact 3-manifolds are considered. 

\begin{definition}\label{geodesic definition}
Let $(M,d)$ be a metric space. Let $x,y\in M$ and let $\eta: I\rightarrow M$ be a path from $\eta(0)=x$ to $\eta(1)=y$. The path $\eta$ is \emph{arc length parametrized} if for any $t,t'\in I$,
\[
\length{M}\left(\left.\eta\right|_{[t,t']}\right)=\length{M}(\eta)~|t'-t|.
\]
The path $\eta$ is a \emph{shortest path} from $x$ to $y$ if 
\[
\length{M}(\eta)=d(x,y).
\]
The path $\eta$ is a \emph{geodesic} from $x$ to $y$ if for any $t,t'\in I$,
\[
d(\eta(t),\eta(t'))=d(x,y)~|t'-t|.
\]
\end{definition}

Every geodesic is a shortest path between its endpoints and is arc length parametrized. 
Thus, every geodesic is Lipschitz with Lipschitz constant equal to its length.

We will primarily be discussing geodesics with reference to metric trees.  Metric trees were originally introduced in \cite{tits1977theorem}. For a selection of results concerning metric trees, see \cite{aksoy2006selection}, \cite{aksoy2010some}, and \cite{mayer1992universal}.

\begin{definition}
A non-empty metric space $T$ is a \textit{metric tree} if for any $x,x'\in T$, there exists a unique arc joining $x$ and $x'$ and there is a geodesic $\eta$ from $x$ to $x'$. A subset $T'\subset T$ of a metric tree is a \textit{subtree} if $T'$ is a metric tree with reference to the metric on $T$ restricted to $T'$.
\end{definition}


\subsection{Wenger and Young's factorization through a metric tree}

We make significant use of the work of Wenger and Young in \cite{Weg} to show the existence of a length minimizing representative, in particular the following factorization:

\begin{theorem}[Theorem 5 in \cite{Weg}]\label{Wenger and Young}
Let $A$ be a quasi-convex metric space with quasi-convexity constant $C$ and with $\pilip{1}(A)=0$. Let furthermore $M$ be a purely 2-unrectifiable metric space. Then every Lipschitz map $f:A\rightarrow M$ factors through a metric tree $T$,
\begin{center}
\begin{tikzcd}
A\arrow[rr, "f"] \arrow[dr, "\psi"'] && M,  \\ 

& T \arrow[ur, "\varphi"']&
\end{tikzcd}
\end{center}
where $\Lip(\psi)=C\Lip(f)$ and $\Lip(\varphi)=1$. 
\end{theorem}

%

 We include some details of their work presently. In the proof of Theorem~\ref{Wenger and Young} in \cite{Weg}, Wenger and Young define the following pseudo-metric on $A$: 
\[
d_f(a,a')\deq\inf\{\length{M}(f\circ c)~|~c\text{ is a Lipschitz path in }A\text{ from }a\text{ to }a'\}
\]
where $a,a'\in A$. The metric tree is then defined as a quotient space $T\deq A/\sim$, where the equivalence relation is given by $a\sim a'$ if and only if $d_f(a,a')=0$. The metric on $T$ is then
\[
d_T([a],[a'])\deq d_f(a,a').
\]
The map $\psi$ is the quotient map, $\psi(a)=[a]$. The original function $f$ is constant on equivalence classes. As such, the map $\varphi([a])=f(a)$ is well-defined.

\section{Length minimizers of homotopy classes in purely 2-unrectifiable metric spaces}

%

\subsection{Building a desirable homotopy}

For the remainder of the paper, let $M$ be a purely 2-unrectifiable metric space with metric $d$. Let $\gamma$ and $\beta$ be Lipschitz homotopic paths. 

We use Theorem~\ref{Wenger and Young} and the definition of the metric tree to fashion a desirable Lipschitz homotopy from the path $\gamma$ to a Lipschitz path $\beta'$ whose length is less than or equal to the length of $\beta$ and whose Lipschitz constant is bounded by the Lipschitz constant of $\gamma$. Furthermore, the desirable homotopy will have Lipschitz constant equal to the Lipschitz constant of $\gamma$. This homotopy will be used to show the existence of a length minimizer in each homotopy class in a purely 2-unrectifiable metric space.

 Let $H:I\times I\rightarrow M$ be a homotopy from $\gamma$ to $\beta$. So, $H|_{I\times\{0\}}=\gamma$ and $H|_{I\times\{1\}}=\beta$. Since $I\times I$ is a geodesic space and Lipschitz simply connected, Theorem~\ref{Wenger and Young} guarantees the Lipschitz map $H$ factors through a metric tree $T$:
\begin{center}
\begin{tikzcd}
{I\times I}\arrow[rr, "H"] \arrow[dr, "\psi"'] && M,  \\ 

& T \arrow[ur, "\varphi"']&
\end{tikzcd}
\end{center}
where $\Lip(\psi)=\lip{H}$ and $\Lip(\varphi)=1$. Note that, since $H|_{\{0\}\times I}=\gamma(0)$ and $H|_{\{1\}\times I}=\gamma(1)$, the restricted maps $\psi|_{\{0\}\times I}=\psi(0,0)$ and $\psi|_{\{1\}\times I}=\psi(1,0)$ are constant. 

Though the map $\psi$ is $\Lip(H)$-Lipschitz, the restriction of the map $\psi$ to $I\times\{0\}$ is at most $\Lip(\gamma)$-Lipschitz:

\begin{lemma}\label{Restriction is path Lipschitz}
$\Lip\left(\psi|_{I\times\{0\}}\right)\leq\Lip(\gamma)$.   
\end{lemma}
\begin{proof}
Let $t,t'\in I$ where $t<t'$. Using the definition of the metric on $T$, 
\begin{eqnarray*}
d_T(\psi(t,0),\psi(t',0)) & = & d_T([(t,0)],[(t',0)]) \\
						  & = & d_H((t,0),(t',0)) \\
						  & = & \inf\{\length{M}(H\circ c)~|~c\text{ is a path from }(t,0)\text{ to }(t',0)\}.
\end{eqnarray*}
Now, selecting the inclusion $c=(\id,0):[t,t']\hookrightarrow I\times I$ which is a Lipschitz path from $(t,0)$ to $(t',0)$, yields that 
\[
d_T(\psi(t,0),\psi(t',0))\leq\length{M}\left(H\circ(\id,0):[t,t']\hookrightarrow M\right)=\length{M}(\gamma|_{[t,t']}).
\]
Since $\gamma$ is Lipschitz, we have the following string of inequalities:
\begin{eqnarray*}
d_T(\psi(t,0),\psi(t',0)) & \leq  & \length{M}(\gamma|_{[t,t']}) \\
						  & \leq & \Lip(\gamma)~|t-t'| .
\end{eqnarray*}

\end{proof}


When defining the new homotopy, our focus in the metric tree $T$ will be $T'\deq\Ima(\psi|_{I\times\{0\}})$, the image of the restriction in Lemma~\ref{Restriction is path Lipschitz}, which is a subtree of $T$. Note that every element of the subtree $T'$ can be written as $[(t,0)]$ for some $t\in I$. The subtree $T'$ has finite diameter bounded by the Lipschitz constant of the path $\gamma$, as is now shown.


\begin{lemma}\label{finite diameter of subtree}
$\diam(T')\leq\Lip(\gamma)$.
\end{lemma}

\begin{proof}
\begin{eqnarray*}
\diam(T') & = & \sup_{[(t,0)],[(t',0)]\in T'}~d_T\left([(t,0)],[(t',0)]\right) \\
		  & = & \sup_{t,t'\in I}~d_H((t,0),(t',0)) \\
		  & = & \sup_{t,t'\in I}~\inf_{c}~\length{M}(H\circ c) \\
		  & \leq & \sup_{t,t'\in I}~\length{M}\left(H\circ(\id,0):[t,t']\rightarrow M\right) \\
		  & = & \sup_{t,t'\in I}~\length{M}\left(\gamma|_{[t,t']}\right) \\
                      & \leq &  \sup_{t,t'\in I}~ \Lip(\gamma)~|t-t'|  \\
		  & = & \Lip(\gamma).
\end{eqnarray*}
\end{proof}

Now, let $\eta:I\rightarrow T$ be the geodesic in $T$ from $\psi(0,0)$ to $\psi(1,0)$. Since $\eta$ is a geodesic, for any $t,t'\in I$,
\begin{equation}\label{gamma geodesic property}
d_T(\eta(t), \eta(t'))=d_T(\psi(0,0), \psi(1,0)) \left|t-t'\right|.
\end{equation}

Define a new path $\beta':I\rightarrow M$ by $\beta'(t)=\varphi\circ\eta(t)$. As will now be shown, the length of $\beta'$ is bounded above by the length of the path $\beta$, the image of $\beta'$ is a subset of the image of $\gamma$, and the Lipschitz constant for $\beta'$ is bounded above by the Lipschitz constant of the initial path $\gamma$.

\begin{lemma}\label{length of beta prime}
$\length{M}(\beta')\leq\length{M}(\beta)$.
\end{lemma}

\begin{proof}
If $\psi(0,0)=\psi(1,0)$, then the geodesic $\eta$ is a constant path, as is the path $\beta'$. The desired result is then immediate.

Assume $\psi(0,0)\neq\psi(1,0)$. Then the geodesic $\eta$ joining these distinct points is non-constant and injective. Let $0=t_0<t_1<t_2<\ldots<t_{n+1}=1$ be a partition of the interval $I$. Then, as will be argued below, there is a partition $0=t_0^*<t_1^*<t_2^*<\ldots<t_{n+1}^*=1$ such that $\beta'(t_i)=\beta(t_i^*)$. 

Now $\eta$ is a geodesic from $\psi(0,0)=\psi(0,1)$ to $\psi(1,0)=\psi(1,1)$ and the map $\psi|_{I\times\{1\}}$ is a path in the metric tree $T$ with the same initial and terminal points as $\eta$. Thus, the image of $\eta$ is a subset of the image of $\psi|_{I\times\{1\}}$. Furthermore, for each $i=1,\ldots,n$, there is a time $t_i^*\in I$ such that
\[
\psi(t_i^*,1)=\eta(t_i)\text{ and }\psi(t,1)\neq\eta(t_i)\text{ for all }t>t_i^*,
\]
that is, $t_i^*$ is the last time the path $\psi|_{I\times\{1\}}$ visits the point $\eta(t_i)$. Thus,
\[
\beta'(t_i)=\varphi(\eta(t_i))=\varphi(\psi(t_i^*,1))=\beta(t_i^*).
\]

Let $i<j$. Suppose $t_i^*\geq t_j^*$. If $t_i^*=t_j^*$, then $\eta(t_i)=\eta(t_j)$, contradicting that the geodesic $\eta$ is injective. Assume $t_i^*>t_j^*$. Then, the resticted path $\psi|_{[t_i^*,1]\times\{1\}}$ begins at $\eta(t_i)$ and ends at $\eta(1)$ and therefore travels through the point $\eta(t_j)$, contradicting that $t_j^*$ is the last time that $\psi|_{I\times\{1\}}$ visits that point $\eta(t_j)$. Therefore, $t_i^*<t_j^*$ for all $i<j$. We thus have attained the desired partition. 

So, for each partition $0=t_0<t_1<t_2<\ldots<t_{n+1}=1$, there exists a partition $0=t_0^*<t_1^*<t_2^*<\ldots<t_{n+1}^*=1$ such that
\[
\sum_{i=0}^{n+1}d\left(\beta'(t_i),\beta'(t_{i+1})\right)=\sum_{i=0}^{n+1}d\left(\beta(t_i^*),\beta(t_{i+1}^*)\right)
\] 
Taking supremum over all  partitions $0=t_0<t_1<t_2<\ldots<t_{n+1}=1$, we arrive at $\length{M}(\beta')\leq\length{M}(\beta)$.
\end{proof}

\begin{lemma}\label{new path has image in initial path} 
$\Ima(\beta')\subset\Ima(\gamma)$.
\end{lemma}

\begin{proof}
The geodesic $\eta$ is a path from $\psi(0,0)$ to $\psi(1,0)$, as is the map $\psi|_{I\times\{0\}}$. Since $\eta$ is a geodesic in the metric tree $T$, the image of the geodesic is a subset of the image of any path with the same initial and terminal points. Thus, $\Ima(\eta)\subset\Ima(\psi|_{I\times\{0\}})$. Therefore,
\[
\Ima(\beta')=\Ima(\varphi\circ\eta)\subset\Ima(\varphi\circ\psi|_{I\times\{0\}})=\Ima(\gamma).
\]
\end{proof}

\begin{lemma}\label{lipschitz constant of beta prime}
$\Lip(\beta')\leq\Lip(\gamma)$.
\end{lemma}

\begin{proof}
Let $t, t'\in I$. Using that $\Lip(\varphi)=1$ as well as (\ref{gamma geodesic property}) and \text{Lemma}~\ref{finite diameter of subtree}, we have the following inequalities:
\begin{eqnarray*}
d(\beta'(t), \beta'(t')) & = & d(\varphi(\eta(t)), \varphi(\eta(t')))  \\
					& \leq & d_T(\eta(t)), \eta(t'))  \\
					& = & d_T(\psi(0,0), \psi(1,0))\left| t-t'\right|  \\
					& \leq & \diam(T')\left|t-t'\right|  \\
					& \leq & \Lip(\gamma)\left|t-t'\right|.
\end{eqnarray*}
\end{proof}

We will now construct a homotopy $H'$ from the initial path $\gamma$ to the new path $\beta'$ which has Lipschitz constant $\Lip(H')=\Lip(\gamma)$. 

Let $t\in I$. There is a geodesic $g_t:I\rightarrow T$ from $g_t(0)=\psi(t,0)$ to $g_t(1)=\eta(t)$ where, for all $s, s'\in I$,
\begin{equation}\label{g_t geodesic property}
d_T(g_t(s), g_t(s')) = d_T(\psi(t,0), \eta(t))\left|s-s'\right|.
\end{equation}

Since points $\psi(0,0)=\eta(0)$ are equal, the geodesic $g_0(s)=\psi(0,0)$ is constant. Similarly, the geodesic $g_1(s)=\psi(1,0)$ is constant. Thus, the function $g:I\times I\rightarrow T$ given by $g(t,s)\deq g_t(s)$ is a homotopy from path $\psi|_{I\times\{0\}}$ to geodesic $\eta$ provided $g$ is Lipschitz. We show that $g$ is a Lipschitz map with Lipschitz constant bounded by $\Lip(\gamma)$ and that the image of $g$ is a subset of the image $\Ima(\psi|_{I\times\{0\}})$.

\begin{lemma}\label{Lipschitz constant of the geodesic homotopy}
$\Lip(g)\leq\Lip(\gamma).$
\end{lemma}

\begin{proof}
Let $(t,s), (t',s')\in I\times I$. 

First, consider $d_T(g_t(s), g_{t'}(s))$. Fix $t$ and $t'$. As $s\in I$ varies, 
\[
D(s)\deq d_T(g_t(s), g_{t'}(s))
\] 
is a function from $I$ to $\R$. By properties of metric trees, there exists $s_0\in I$ such that the restriction $D|_{[0,s_0]}$ is decreasing and the restriction $D|_{[s_0,1]}$ is increasing. 
Thus, the maximum of the function $D$ occurs when $s=0$ or $s=1$. Now, by Lemma~\ref{Restriction is path Lipschitz},
\begin{eqnarray*}
D(0) ~ = ~ d_T(g_t(0), g_{t'}(0)) & = & d_T(\psi(t,0), \psi(t',0)) \\
						      & \leq & \Lip(\gamma)\left|t-t'\right|. 
\end{eqnarray*}
Also, by (\ref{gamma geodesic property}) and Lemma~\ref{finite diameter of subtree},
\begin{eqnarray*}
D(1) ~ = ~ d_T(g_t(1), g_{t'}(1)) & = & d_T(\eta(t), \eta(t')) \\
							      & = & d_T(\eta(0), \eta(1))\left|t-t'\right| \\
							      & \leq & \diam(T')\left|t-t'\right| \\
							      & \leq & \Lip(\gamma)\left|t-t'\right|.
\end{eqnarray*}
Therefore, for any $s\in I$,
\[
d_T(g_t(s), g_{t'}(s)) = D(s) \leq \Lip(\gamma)\left|t-t'\right|.
\]

Now, consider the value $d_T(g_t(s), g_{t}(s'))$. Since $\eta$ is a geodesic from $\psi(0,0)$ to $\psi(1,0)$ and $T'=\Ima(\psi|_{I\times\{0\}})\subset T$ is a subtree containing these points, $\eta(t)\in T'$. Thus, by (\ref{g_t geodesic property}) and Lemma~\ref{finite diameter of subtree},
\begin{eqnarray*}
d_T(g_t(s), g_t(s')) & = & d_T(\psi(t,0), \eta(t))\left|s-s'\right| \\
				    & \leq & \diam(T')\left|s-s'\right| \\
				    & \leq & \Lip(\gamma)\left|s-s'\right|.
\end{eqnarray*}

Therefore, we conclude that
\begin{eqnarray*}
d_T(g_t(s), g_{t'}(s')) & \leq & d_T(g_t(s), g_{t'}(s)) + d_T(g_{t'}(s), g_{t'}(s')) \\
					& \leq & \Lip(\gamma)\left|t-t'\right| + \Lip(\gamma)\left|s-s'\right| \\
					& = & \Lip(\gamma)~d^1((t,s), (t',s')).
\end{eqnarray*}

\end{proof}

\begin{lemma}\label{g maps into T prime} 
$\Ima(g)\subset \Ima(\psi|_{I\times\{0\}})$.
\end{lemma}
\begin{proof}
Let $(t,s)\in I\times I$. The path $g_t$ is a geodesic from $\psi(t,0)$ to $\eta(t)$. Since $\eta$ is a geodesic from $\psi(0,0)$ to $\psi(1,0)$ and $T'=\Ima(\psi|_{I\times\{0\}})$ is a subtree containing these points, $\eta(t)\in T'$. Thus, since $g_t$ is a geodesic between points $\psi(t,0), \eta(t)\in T'$ and $T'$ is a subtree, $g_t(s)\in T'$ for all $s\in I$. Thus, $\Ima(g)\subset \Ima(\psi|_{I\times\{0\}})$. 
\end{proof}

We are now ready to define the new homotopy $H':I\times I\rightarrow M$ by $H'(t,s)\deq\varphi\circ g(t,s)$. The function $H'$ is indeed a homotopy from $\gamma$ to $\beta'$ as
\begin{center}
$\begin{array}{ccccccc}
H'(t,0) & = & \varphi(g_t(0)) & = & \varphi(\psi(t,0)) & = & \gamma(t),  \\
H'(t,1) & = & \varphi(g_t(1)) & = & \varphi(\eta(t)) & = & \beta'(t), \\
H'(0,s) & = & \varphi(g_0(s)) & = & \varphi(\psi(0,0)) & = & \gamma(0),  \\
H'(1,s) & = & \varphi(g_1(s)) & = & \varphi(\psi(1,0)) & = & \gamma(1). 
\end{array}$
\end{center}
Moreover, since $\Lip(\varphi)=1$ and, by Lemma~\ref{Lipschitz constant of the geodesic homotopy}, for $(s,t), (s',t')\in I\times I$,
\begin{eqnarray*}
d(H'(t,s), H'(t',s')) & = & d(\varphi(g(t,s)), \varphi(g(t',s'))) \\
				 & \leq & d_T(g(t,s), g(t',s')) \\
				 & \leq & \Lip(\gamma)~d^1((t,s), (t',s')).
\end{eqnarray*}
Therefore, $\Lip(H')\leq\Lip(\gamma)$. In fact, since $\Lip(H'|_{I\times\{0\}})=\Lip(\gamma)$, we have that $\Lip(H')=\Lip(\gamma)$. Also, an immediate consequence of Lemma~\ref{g maps into T prime} is that $\Ima(H')\subset\Ima(\gamma)$.

We have thus defined a Lipschitz homotopy with all of the desired properties, which are collected in the following lemma.

\begin{lemma}\label{Desirable homotopy}
Let $M$ be a purely 2-unrectifiable space. Given Lipschitz paths $\gamma:I\rightarrow M$ and $\beta:I\rightarrow M$ that are  homotopic rel endpoints, there exists a Lipschitz map $H':I\times I\rightarrow M$ and a Lipschitz path $\beta':I\rightarrow M$ such that
\begin{itemize}
\item the map $H'$ is a homotopy from $\gamma$ to $\beta'$,
\item $\Lip(H')=\Lip(\gamma)$,
\item $\Ima(H')\subset\Ima(\gamma)$,
\item $\Lip(\beta')\leq \Lip(\gamma)$,
\item $\Ima(\beta')\subset\Ima(\gamma)$, and
\item $\length{M}(\beta')\leq\length{M}(\beta)$.
\end{itemize}
\end{lemma}

\subsection{Finding the length minimizer of a homotopy class}

We now prove the primary result of the paper: the existence of a length minimizer in any homotopy class of paths in a purely 2-unrectifiable metric space. We use Lemma~\ref{Desirable homotopy} to fashion a sequence of Lipschitz paths in a given homotopy class, as well as associated homotopies, that have a uniform bound on their Lipschitz constants and then apply Arzel\`{a}-Ascoli theorem to find the length minimizer. 


\begin{theorem}\label{existence of length minimizer}
Let $M$ be a purely 2-unrectifiable metric space. For any homotopy class $[\gamma]$ of Lipschitz paths in $M$, there exists a length minimizing Lipschitz path $\core{\gamma}\in[\gamma]$ where
\[
\length{M}(\core{\gamma})=\inf\{\length{M}(\gamma)~|~\gamma\in[\gamma]\}.
\]
\end{theorem}

\begin{proof}

Let $M$ be a purely 2-unrectifiable metric space. Let $[\gamma]$ be a homotopy class of Lipschitz paths and define $\lmin\deq\inf\{\length{M}(\gamma)~|~\gamma\in[\gamma]\}$ to be the infimum of all lengths of paths in $[\gamma]$. 

For each natural number $n$, let $\gamma_n\in[\gamma]$ be a Lipschitz path such that $\length{M}(\gamma_n)\leq\lmin+\frac{1}{n}$. Furthermore, since $\gamma_1$ is homotopic rel endpoints to $\gamma_n$, via Lemma~\ref{Desirable homotopy}, we can assume that $\Lip(\gamma_n)\leq\Lip(\gamma_1)$ and $\Ima(\gamma_n)\subset\Ima(\gamma_1)$. Additionally, there is a homotopy $H_n:I\times I\rightarrow M$ from $\gamma_1$ to $\gamma_n$ such that $\Lip(H_n)=\Lip(\gamma_1)$ and $\Ima(H_n)\subset\Ima(\gamma_1)$.

Now, for any $n\in\NN$, $\Lip(\gamma_n)\leq\Lip(\gamma_1)$. Since the images of the paths in the sequence $(\gamma_n)$ are subsets of the compact set $\Ima(\gamma_1)$, by Arzel\`{a}-Ascoli theorem, there exists a subsequence $(\gamma_{n_k})$ that uniformly converges to a Lipschitz path $\core{\gamma}$. By lower semi-continuity of the length measure, $\length{M}(\core{\gamma})\leq\liminf_k\length{M}(\gamma_{n_k})$. In fact, due to how the sequence $(\gamma_n)$ was selected, $\length{M}(\core{\gamma})\leq\lmin$. 

We now want to show that $\core{\gamma}\in[\gamma]$. Associated to the subsequence $(\gamma_{n_k})$, there is a sequence of homotopies $(H_{n_k})$ such that $\Lip(H_{n_k})=\Lip(\gamma_1)$ for each homotopy in the sequence. 
Since $\Ima(H_{n_k})\subset\Ima(\gamma_1)$ for each $n_k$ and $\Ima(\gamma_1)$ is compact, by Arzel\`{a}-Ascoli theorem, there exists a subsequence $(H_{n_{k_j}})$ that converges uniformly to a Lipschitz map $\core{H}:I\times I\rightarrow M$. 

Now, $\core{H}|_{I\times\{0\}}=\gamma_1$ since $H_{n_{k_j}}|_{I\times\{0\}}=\gamma_1$ for all $n_{k_j}$. Also, since the paths $H_{n_{k_j}}|_{I\times\{1\}}=\gamma_{n_{k_j}}$ converge uniformly to $\core{\gamma}$, then $\core{H}|_{I\times\{1\}}=\core{\gamma}$. So, the map $\core{H}$ is a homotopy from $\gamma_1$ to $\core{\gamma}$. Therefore, $\core{\gamma}\in[\gamma]$ and thus $\length{M}(\core{\gamma})=\lmin$.

\end{proof}

\subsection{Consequences of the existence of a length minimizer}

A length minimizer $\core{\gamma}\in[\gamma]$ can be thought of as the core of the homotopy class $[\gamma]$ where the extraneous branches of the paths in the class have been pruned in the sense that the image of $\core{\gamma}$ is a subset of the image of any path contained in $[\gamma]$, as is now shown. A consequence of Theorem~\ref{uniqueness of core} is that a length minimzer for a homotopy class is unique up to reparametrization.

\begin{theorem}\label{uniqueness of core}
Let $M$ be a purely 2-unrectifiable metric space and let $[\gamma]$ be a homotopy class of Lipschitz paths in $M$ with length minimizer $\core{\gamma}\in[\gamma]$. Additionally, assume that the length minimzer $\core{\gamma}$ is arc length parametrized. Let $\gamma:I\rightarrow M$ be a Lipschitz path that is  homotopic rel endpoints to $\core{\gamma}$. Then the Lipschitz path $\core{\gamma}'$ produced by Lemma~\ref{Desirable homotopy} is equal to the length minimizer $\core{\gamma}$. Furthermore, the image of a length minimizer $\core{\gamma}$ is a subset of the image of $\gamma$, that is,
\[
\Ima(\core{\gamma})\subset\Ima(\gamma).
\]
\end{theorem}

\begin{proof} 
Let $H:I\times I\rightarrow M$ be a Lipschitz homotopy from $\gamma$ to $\core{\gamma}$. By Theorem~\ref{Wenger and Young}, the map $H$ factors through a metric tree $T$:
\begin{center}
\begin{tikzcd}
{I\times I}\arrow[rr, "H"] \arrow[dr, "\psi"'] && M.  \\ 

& T \arrow[ur, "\varphi"']&
\end{tikzcd}
\end{center}

We now show that the path $\psi|_{I\times\{1\}}$ in the metric tree $T$ is the geodesic from $\psi(0,1)$ to $\psi(1,1)$. Let $t,t'\in I$ where $t<t'$ and let $c$ be a Lipschitz path in $I\times I$ from $(t,1)$ to $(t',1)$. Since $c$ is homotopic to the inclusion $(\id,1):[t,t']\hookrightarrow I\times I$, the paths $H\circ c$ and $H\circ(\id,1)=\core{\gamma}|_{[t,t']}$ are homotopic. Since $\core{\gamma}$ is the length minimizer for $[\gamma]$, the restriction $\core{\gamma}|_{[t,t']}$ is also a length minimizer in its homotopy class. Thus, $\length{M}(\core{\gamma}|_{[t,t']})\leq\length{M}(H\circ c)$. Therefore, by the definition of the metric on $T$, 
\[
d_T(\psi(t,1),\psi(t',1))=\length{M}\left(\core{\gamma}|_{[t,t']}\right)
\]
and in particular, $d_T(\psi(0,1),\psi(1,1))=\length{M}(\core{\gamma}).$ Now, since $\core{\gamma}$ is arc length parametrized,
\begin{eqnarray*}
d_T(\psi(t,1),\psi(t',1)) & = & \length{M}\left(\core{\gamma}|_{[t,t']}\right) \\
					 & = & \length{M}(\core{\gamma})~|t'-t| \\
					 & = & d_T(\psi(0,1),\psi(1,1))~|t'-t|.
\end{eqnarray*}
Therefore, the path $\psi|_{I\times\{1\}}$ is indeed the geodesic from $\psi(0,1)$ to $\psi(1,1)$. 

From the argument of Lemma~\ref{Desirable homotopy}, the path $\core{\gamma}'$ is equal to the geodesic in $T$ from $\psi(0,0)=\psi(0,1)$ to $\psi(1,0)=\psi(1,1)$ post-composed by $\varphi$. As the geodesic in discussion is $\psi|_{I\times\{1\}}$, we have that for all $t\in I$,
\[
\core{\gamma}(t)=\varphi\circ\psi(t,1)=\core{\gamma}'(t).
\]

That  $\Ima(\core{\gamma})\subset\Ima(\gamma)$ follows quickly from $\core{\gamma}$ factoring through a geodesic segment. Indeed, for $t\in I$, the point $\psi(t,1)\in T$ is in the geodesic segment $\Ima(\psi|_{I\times\{1\}})$ connecting $\psi(0,0)=\psi(0,1)$ to $\psi(1,0)=\psi(1,1)$. As the image $\Ima(\psi|_{I\times\{0\}})\subset T$ is a subtree containing these points, the geodesic segement $\Ima(\psi|_{I\times\{1\}})\subset\Ima(\psi|_{I\times\{0\}})$ is a subset of the subtree. Hence, there exists $t'\in I$ such that $\psi(t,1)=\psi(t',0)$. Therefore,
\[
\core{\gamma}(t)=\varphi\circ\psi(t,1)=\varphi\circ\psi(t',0)=\gamma(t').
\]
Thus, $\Ima(\core{\gamma})\subset\Ima(\gamma)$ as desired.

\end{proof}

Of note, in the proof of Theorem~\ref{uniqueness of core} we have shown that given an arc length parametrized length minimizer and any homotopy of the length minimzer, the length minimzer factors through a geodesic segment in the metric tree generated by the homotopy via Theorem~\ref{Wenger and Young}.

An immediate consequence of Theorem~\ref{uniqueness of core} is that for every point in a purely 2-unrectifiable metric space, only the trivial class in the first Lipschitz homotopy group can be represented by a loop within each neighborhood of the point, as is now shown.


\begin{corollary}\label{heisenberg non-singular points}
Let $M$ be a purely 2-unrectifiable metric space and let $x\in M$. Let $[\alpha]\in\pilip{1}(M,x)$ be a homotopy class of loops based at $x$ such that for every open neighborhood $U\subset M$ of the point $x$, there exists a Lipschitz loop $\alpha_U\in[\alpha]$ based at $x$ whose image is contained in $U$. Then $[\alpha]$ is the trivial homotopy class, $[\alpha]=[x]$.
\end{corollary}

\begin{proof}
Let $[\alpha]\in\pilip{1}(M,x)$ be a homotopy class of loops based at $x$ such that for every open neighborhood $U\subset M$ of the point $x$, there exists a Lipschitz loop $\alpha_U\in[\alpha]$ based at $x$ whose image is contained in $U$. Then, by Theorem~\ref{existence of length minimizer}, $[\alpha]$ has a length minimizer $\core{\alpha}$ and, by Theorem~\ref{uniqueness of core}, the image of the length minimizer $\core{\alpha}$ is a subset of every neighborhood $U$ of $x$. Therefore, $\core{\alpha}$ is the constant loop at $x$ and thus $[\alpha]=[x]$.
\end{proof}

Using the wording of \cite{perry2024universal}, Corollary~\ref{heisenberg non-singular points} says that every point in a purely 2-unrectifiable supports only trivial local representation. In the language of \cite{Bog}, every point in a purely 2-unrectifiable metric space is non-singular. The harmonic archipelago is an instructive example of a space wherein not every point is non-singular. See Example 1.1 in \cite{Bog}.

To define the metric on the universal Lipschitz path space, it is sufficient that for every point in the underlying metric space, only the trivial class in the first Lipschitz homotopy group can be represented by a loop within each neighborhood of the point. As such, Corollary~\ref{heisenberg non-singular points} is sufficient to define the metric on the universal Lipschitz path space over a purely 2-unrectifiable metric space such as the Heisenberg group $\H^1$.


\end{document}